\documentclass[12pt,reqno]{amsart}

\usepackage{amsmath}
\usepackage{amssymb}

\numberwithin{equation}{section}

\newcommand{\C}{{\mathbb{C}}}

\newcommand{\R}{{\mathbb{R}}}

\newcommand{\Aa}{{\mathcal{A}}}

\newcommand{\Cc}{{\mathcal{C}}}

\newcommand{\Gg}{{\mathcal{G}}}
\newcommand{\Hh}{{\mathcal{H}}}

\newcommand{\Mm}{{\mathcal{M}}}

\newcommand{\euler}{\mathcal{X}_{0}}
\newcommand{\imag}{\mathrm{\mathbf{i}}}

\newcommand{\sst}{\scriptscriptstyle}

\newcommand{\ms}[1]{\mathsf{#1}}

\newcommand{\mf}[1]{\mathfrak{#1}}

\newcommand{\abs}[1]{\left\lvert #1 \right\rvert}
\newcommand{\pair}[1]{\left\langle #1 \right\rangle}

\newcommand{\eqst}[1]{\begin{equation*} #1 
                      \end{equation*}}
\newcommand{\eq}[1]{\begin{equation} #1
                    \end{equation}}

\theoremstyle{plain}
\newtheorem{thm}{Theorem}[section]
\newtheorem{lem}[thm]{Lemma}
\newtheorem*{lemma*}{Lemma}

\theoremstyle{definition}

\theoremstyle{definition}
\newtheorem{defn}{Definition}

\theoremstyle{remark}
\newtheorem*{question*}{Question}

\DeclareMathOperator{\Map}{Map}

\DeclareMathOperator{\grad}{grad}

\usepackage{hyperref}
\hypersetup{
   colorlinks = true,
    linkcolor = blue,
    citecolor = blue,   
}

\allowdisplaybreaks

\begin{document}

\title[Symplectic vortex equations for K\"ahler cones over Sasakian manifolds]{Symplectic vortex equations for K\"ahler cones over Sasakian manifolds}

\author[V. Thakre]{Varun Thakre}

\address{International Centre for Theoretical Sciences (ICTS-TIFR), Hesaraghatta, Hobli, Bengaluru 560089, India}

\email{varun.thakre@icts.res.in}


\date{\today}


\begin{abstract}

We obtain a Hitchin-Kobayashi-type correspondence for symplectic vortex equations, with the target a K\"ahler cone over a compact Sasakian manifold. We show that the correspondence reduces to studying the existence and uniqueness of Kazdan-Warner equations. Using this, we construct a map between the moduli space of solutions to the symplectic vortex equations and effective divisors.

\end{abstract}

\maketitle

\section{Introduction}

In this article we study a gauged, non-linear $\sigma$-model, wherein a principal $G$-bundle over a closed K\"ahler manifold $(X, \omega_{\sst X})$ is mapped to a K\"ahler manifold $(M, \omega_{\sst M})$, also called the \emph{target manifold}. The said equations nothing but the \emph{symplectic vortex equations}, which are an analogue of vortex equations, for maps taking values in $M$. They were discovered independently by K. Cieliebak, A. Gaio and D. Salamon \cite{ciel-giao-salamon00} and by I. Mundet i Riera \cite{riera99}. The latter obtained a general form of a Hitchin-Kobayashi-type correspondence for the equations over any compact K\"ahler manifold. The correspondence relates the space of solutions the K\"ahler vortex equations upto complex and real gauge transformations. This coincides with the usual notion of ``stability" that arises in the construction of the algebraic moduli spaces, using Geometric Invariant Theory (GIT). 

The correspondence, in general, is not easy to evaluate, even in the Abelian set-up. However, when $M$ is a K\"ahler cone over a compact Sasakian manifold, we show that the stability criterion can be reduced to existence and uniqueness of solutions to Kazdan-Warner equations. The latter is a well-known, second order PDE, that originally appeared in the problem of finding Riemannian metrics with prescribed scalar curvature in two dimensions \cite{kazdan-warner74}. We show that the stability criterion coincides with that of the usual vortex equations. This enables us to construct a map from the moduli space of solutions to the space of effective divisors on $X$.

A \emph{twisted} version the equations was recently studied by the author and I. Biswas \cite{indranil-varun18} in four dimensions, when the target $M$ is a hyperK\"ahler cone over a compact 3-Sasakian manifold. The authors constructed a map from the moduli space of solutions to the equations, to the space of effective divisors. A natural question to ask is \emph{how general is the correspondence between the solutions to gauged, non-linear $\sigma$-models and effective divisors}? The question motivates our study of K\"ahler vortex equations in the this article.

\section{Sasakian Geometry}

We begin by giving a brief introduction to Sasakian geometry. The results stated in this section are mostly elementary in nature. A more detailed discussion can be found, for instance, in \cite{boyer98} and \cite{sparks10}.

Sasakian geometry is best described as an odd dimensional analogue of K\"ahler geometry. It lies at the interface of CR, contact and Riemannian geometry. 
Perhaps the most straight forward definition of a Sasakian manifold, due to C. Boyer \cite{boyer98}, is as follows
\begin{defn}
Let $(S, g_{\sst S})$ be a $2n+1$-dimensional Riemannian manifold. Consider the metric cone over $S$, with the \emph{warped metric}
\eqst{
C(S)\, := \, \R^+ \times S, \, \, \, \, \, \, g_{\sst C(S)} := \, dr^2 + r^2\, g_{\sst S}
}
where $r$ is the radial co-ordinate. A vector field $\xi \in \Gamma(TS)$ is called a \emph{Sasakian structure} on $S$ if the metric cone $C(S)$, with the warped metric, is a K\"ahler manifold.
\end{defn}

\noindent The vector field $\xi$ is also called the \emph{Reeb vector field} (or the \emph{characteristic vector field}). The simplest example of a Sasakian manifold is the odd-dimensional sphere $S^{2n-1} \subset \C^n$, with a flat-metric induced from $\C^n$. The Reeb vector field at a point $p \in S^{2n-1}$ is given by $\xi|_p = -\imag p$. This is called the standard Sasakian structure on the odd-dimensional sphere. It is possible to construct more general Sasakian structures by twisting the standard one.

\noindent Let $\eta$ denote the metric-dual 1-form of the Reeb vector field. Then the K\"ahler 2-form over the cone (with obvious identifications) can be written as
\eq{
\label{eq: kahler form on cone}
\omega_{\sst C(S)} \, = \, dr^2 \wedge \eta \, + \, r^2 d\eta.
}
\noindent Suppose that a compact Lie group $G$ acts on $S$ isometrically, preserving the contact structure $\xi$ and therefore also the contact 1-form $\eta$. Let $\mf{g}$ be the Lie algebra of $G$ and denote by $K^S_{\beta}$ the fundamental vector field due the element $\beta \in \mf{g}$. The \emph{Sasakian moment map} for the $G$-action on $S$ is defined to be the $G$-equivariant map \cite{ornea01} 
\eqst{
\gamma: S \longrightarrow \mf{g}^{\ast}, \, \, \, \, \, \, \pair{\gamma(s), \beta} = \eta\left(K^S_{\beta} \right|_s) \, \, \, \, \, \, \text{where} \, \, \, \, \, \, K^S_{\beta}|_s = \frac{d}{dt} \left(\exp(t\beta)\cdot s \right)|_{t=0} .
}
\noindent It is clear that the Sasakian action of $G$ on $S$ lifts to a Hamiltonian action on the cone, preserving $\omega_{\sst C(S)}$. The action also commutes with the homothetic action of $\R^+$ on $C(S)$. There is a canonical, globally defined K\"ahler potential on $C(S)$, which is given by $\rho_0 (r, s) \, = \, \frac{1}{2}\, r^2$ (see \cite{sparks10}). Moreover, the \emph{Euler vector field} $\euler := \, r\partial r = \grad\rho_0\,$ preserves the 1-form $\eta$. It follows from \eqref{eq: kahler form on cone} that the moment map for the $G$-action on the cone can be written as
\eqst{
\mu: \, C(S) \longrightarrow \mf{g}^{\ast}, \, \, \, \, \, \, \pair{\mu(r,s), \beta} \, =\, \rho_0(r,s) \, \pair{\gamma(s), \beta}.
}

\noindent Given the relation between the Sasakian and the K\"ahler geometry via the cone construction, it is natural to ask if there exists a reduction scheme for Sasakian manifolds, that lifts to the K\"ahler reduction on the cone. This was studied by G. Grantcharov and L. Ornea \cite{ornea01} for the case when zero is a regular value of the Sasakian moment map and by O. Dr\v{a}gulete and L. Ornea \cite{ornea06} for the non-zero value. Since the map $\mu$ is homogeneous with respect to the homothetic $\R$-action, the Marsden-Weinstein quotient of $C(S)$ is once again a K\"ahler cone over a compact Sasakian manifold. The latter corresponds to the Sasakian reduction of $S$ with respect to the $G$-action \cite{ornea01}, \cite{ornea06}.

\section{Symplectic vortices}

In this section, we briefly introduce the symplectic vortex equations, which are the main focus of our study. We refer to \cite{ciel-giao-salamon00} and \cite{riera99} for a more detailed discussion on the same.

Let $(X, \omega_{\sst X})$ be a K\"ahler manifold and $(M, \omega_{\sst M})$ be a symplectic manifold, equipped with a Hamiltonian action of a Lie group $G$. Let $\mu$ be the associated moment map 
\eqst{
\mu \, : \, M \longrightarrow \mf{g}^{\ast} \cong \mf{g}.
}
Let $P \rightarrow X$ be a principal $G$-bundle over $X$. We emphasize that the complex structure on $X$ is fixed. Denote by $\Map(P , M)^{G}$ the space of smooth $G$-equivariant maps from $P$ to $M$ and by $\Aa(P)$ the space of connections on $P$. Define the \emph{configuration space}
\eq{
\label{eq: configuration space}
\Cc \,:= \, \Map\,(P, M)^{G} \times \Aa(P).
}
There is a (right) action of the infinite-dimensional \emph{gauge group} $\Gg\, := \, \Map \, (P, G)^{G}$ on $\Cc$. 

\noindent An equivariant map $u$ and a connection $A$, together determine a $G$-equivariant map
$K^M_{A}(u)\,:\, TP \,\rightarrow\, TM$ as follows. For any $v \,\in\, T_p P$, take 
\eqst{
K^M_{A}(v)(u) \,:=\, K^M_{A(v)}(u(p)) \,\in\, T_{u(p)}M.
}
Note that $A(v)\, \in\, \mf{g}$. The differential of $u$ is also $G$-equivariant map. We define the covariant derivative of $u$ with respect to a connection $A$ to be the one-form $D_{A}u \in \Omega^1(P, u^{\ast}TM)^{G}$
\eqst{
D_{A}u \,=\, du + K^M_{A}(u)\, .
}
This is an equivariant, horizontal one-form on $P$. Indeed, for any $\beta \in \mf{g}$
\eqst{
D_A u \left(K^{P}_{\beta} \right) = du\left(K^{P}_{\beta} \right)\, + \, K^M_{A\left(K^{P}_{\beta} \right)}(u) = - K^M_{\beta}(u) \, + \, K^M_{\beta}(u) = 0.
}
It therefore descends to a one-form on $X$ with values in the pull-back bundle $(u^{\ast}TM)/G$ over $X$. Denote by $\overline{\partial}_{\ms{A}}u$, the $(0,1)$-part of this 1-form, meaning
\eqst{
\overline{\partial}_{A}u \,=\, \frac{1}{2} \left(D_{A}u \, + \, I \circ D_{A}u \circ \widetilde{I}_X \right),
}
where $\widetilde{I}_X$ is the lift of the complex structure $I_X$ to the horizontal subspace $\Hh_{A} \subset TP$ and $I$ is the almost-complex structure on $M$, compatible with $\omega_{\sst M}$.

For a pair $(u, A) \in \Cc$, the symplectic vortex equations are a system of equations, satisfying
\eq{\label{eq: symplectic vortices}
\left\{
    \begin{array}{lcl}
     \overline{\partial}_{A}u = 0 \\      
     \Lambda_{\omega_{\sst X}} F_{A} \, - \, \mu \circ u \, + \tau = 0 \\[1mm]
     F^{0,2}_{A} = 0
    \end{array}
  \right.}
where $F_{A}$ is the curvature of $A$ and $\tau \in Z(\mf{g})$ is an element in the center of the Lie algebra. The equations are invariant under the action of the infinite-dimensional gauge group.

\section{A Hitchin-Kobayashi correspondence}

In this section, we establish a Hitchin-Kobayashi type correspondence for the case when $G = {\rm U}(1)$ and $M$ is a K\"ahler cone over a compact Sasakian manifold $(S, g_{\sst S})$. Assume that there is an isometric action of ${\rm U}(1)$ on $S$, preserving the Sasakian structure. Then the action lifts to a Hamiltonian action on $M = C(S)$, commuting with the homothetic $\R^+$ action. For the rest of the article, we fix a principal ${\rm U}(1)$-bundle $P \rightarrow X$.

The equations \eqref{eq: symplectic vortices} can be viewed as Kempf-Ness theory in infinite dimensions as follows. To begin with, in what follows, we implicitly assume the completion of the configuration space and the gauge group in the appropriate Sobolev $(k,p)$-norm, with $k-\frac{2n}{p}>0$, where $n$ is the complex dimension of $X$. For details on Sobolev completion of maps between manifolds, we refer to Subsection 4.1, Appendix B of \cite{wehrheim}. 

The configuration space carries a natural K\"ahler structure, induced by the complex structures $I_X$ on $X$ and $I$ on $M$ (see \cite{riera99}, Sec. 2.3) and the $L^2$-metric. The (right) action of the gauge group preserves the K\"ahler structure. The second equation of \eqref{eq: symplectic vortices} can be interpreted as a moment map for the holomorphic action of $\Gg$ on the configuration space $\Cc$. Moreover, the action extends naturally to an action of its complexification $\Gg^{\C}\, = \, \Map \, (X, \, \C^{\ast})$, with respect to the induced complex structure. Consider the complex sub-variety
\eqst{
\Cc^{1,1} := \, \Map(P, M)^{{\rm U}(1)} \times \Aa^{1,1}(P)
}
where $ \Aa^{1,1}(P) \subset \Aa(P)$ is the subspace of all connections $A$ such that $F_A^{0,2} = 0$. Then the action of $\Gg^{\C}$ descends to an action on $\Cc^{1,1}$. The question of stability, which is the essense of Hitchin-Kobayashi correspondence, narrows down to asking when does a $\Gg^{\C}$-orbit in $\Cc^{1,1}$ intersect the zero of the infinite-dimensional moment map. Notice that the first and third equations of \eqref{eq: symplectic vortices} are anyway invariant under $\Gg^{\C}$.  This is a common paradigm in gauge theory, pioneered by M. Atiyah and R. Bott \cite{atiyah-bott82}. Since its inception, the technique has found applications in various other contexts, most notably by Donaldson \cite{donaldson83}, \cite{donaldson85} and by Uhlenbeck and Yau \cite{uhlenbeck-yau86} to relate stable vector bundles over complex manifolds with Hermitian-Einstein vector bundles and in the study of various other gauge theoretic equations \cite{oscar94}, \cite{oscar-bradlow95}, \cite{teleman96}.

Define $\Hh^{1,1} \subset \Cc^{1,1}$ to be the subset given by
\eqst{
\Hh^{1,1} \, := \, \{(u,A) \in \Cc^{1,1} \, | \, \overline{\partial}_{A}u = 0. \}
}
Let $\Phi \, : \, \Cc \longrightarrow \Map(X, \imag \R)$ denote the moment map for the holomorphic action of the gauge group, given by
\eqst{
\Phi(u,A) \, = \, \Lambda_{\omega_{\sst X}}F_{A} \, - \, \imag \mu\circ u \, + \imag\tau, ~~~~ \text{where} ~~~~ \tau \in \R.
}
The moduli space of solutions to \eqref{eq: symplectic vortices} is a K\"ahler sub-manifold of $\Phi^{-1}(0)/\Gg$, given by
\eqst{
\Mm^{SV}(g_{\sst X}, M) \, := \, \left(\Hh^{1,1} \, \cap \, \Phi^{-1}(0)\right)/\Gg.
}

Denote by $F_0 \subset M$ the fixed-point set of the ${\rm U}(1)$-action on $M$ and consider the dense open set 
\eqst{
\Hh^{ss} \, := \, \{(u,A) \in \Hh^{1,1} \, | \, u(P) \not\subset F_0 \}.
}

\begin{thm}[\textbf{Hitchin-Kobayashi correspondence}] 
\label{thm: main thm 1}
Let $(u,A) \in \Hh^{1,1}$ and assume that $\tau > \frac{2\pi}{\text{Vol}(X)}\deg_{\omega_{\sst X}} P$. Then the moduli space $\Mm^{SV}(g_{\sst X})$ is non-empty and has a holomorphic description
\eqst{
\Mm^{SV}(g_{\sst X}) \cong \Hh^{ss}/\Gg^{\C}.
}
\end{thm}

\begin{proof}
To begin with, note that the first and the third equation of \eqref{eq: symplectic vortices} are invariant under $\Gg^{\C}$. Consequently, the action of $\Gg^{\C}$ on $\Cc$ descends to an action on $\Hh^{ss}$. So in order to prove the statement, we need to establish the existence of a complex gauge transformation $g \in \Gg^{\C}$ such that $\Phi(g\cdot u, g\cdot A) = 0$.

Consider an element $e^f \in \Gg^{\C}$, where $f: X \rightarrow \C$. If $f$ is purely imaginary, then $e^f \in \Gg$. Since the equations \eqref{eq: symplectic vortices} are gauge-invariant, we consider the case when $f$ is real. Now $\Gg^{\C}$ acts on $\Hh^{ss}$ as
\eqst{
e^f \cdot (u, A) \longmapsto \left(e^f u, A + \partial\overline{f} - \overline{\partial} f \right).
}
We have $F_{e^f\cdot A}\, = \, F_A + \overline{\partial}\partial f - \partial\overline{\partial}f$. To understand the action on $u$, recall that since $M$ is a cone over a Sasakian, it has a natural homothetic action of $\R^+$. Since the ${\rm U}(1)$-action comutes with the homothetic $\R^+$-action, we also have $\mu (e^f u) = e^{2f} \mu\circ u$. Therefore,
\eqst{
\Phi(e^f u, \, e^f\cdot A) \, = \, \Lambda_{\omega_{\sst X}}(\overline{\partial}\partial - \partial\overline{\partial})f \, - \, \imag e^{2f}\mu\circ u \, + \, \Lambda_{\omega_{\sst X}} F_A \, + \, \imag \tau.
}
To find a complex gauge transformation that preserves the zero-level set of $\Phi$, we need to solve
\eqst{
 \, \Lambda_{\omega_{\sst X}}(\overline{\partial}\partial - \partial\overline{\partial})f \, - \, \imag e^{2f}\mu\circ u \, + \, \Lambda_{\omega_{\sst X}} F_A \, + \, \imag \tau \, = \, 0.
}
Using the fact that $\Lambda_{\omega_{\sst X}} \overline{\partial}\partial f \, = \, -\frac{\imag}{2} \Delta_X f$ and $\Lambda_{\omega_{\sst X}} \partial\overline{\partial} f \, = \, \frac{\imag}{2} \Delta_X f$, where $\Delta_X$ is the positive-definite Laplacian on $X$, we can re-write the above equation as
\eqst{
\Delta_X f \, + \, e^{2f} \mu\circ u \, = \, \left(\tau \, - \, \imag\Lambda_{\omega_{\sst X}} F_A \right) \, := \, w.
}
Observe that $\mu\circ u: P \rightarrow \R$ is a ${\rm U}(1)$-invariant map and therefore descends to a real-valued function over $X$. Let us denote this function by $B$. We therefore have
\eq{
\label{eq: sympl. vort. kazdan-warner}
\Delta_X f \, + \, e^{2f} B(x) \,= \, w.
}
We now recall the following result of Kazdan and Warner \cite{kazdan-warner74}

\begin{lem}
\label{lem: kazdan-warner}
Let $X$ be a compact Riemannian manifold and let $B$ and $w$ be smooth functions on $X$. Suppose that $B$ is positive outside of a measure zero set and $\int_X w \,>\, 0$. Let $\Delta_X $ be the positive definite Laplacian on $X$. Then the equation 
\eqst{
\Delta_X f \, + \, B(x) e^{2f} \, - \, w \,=\, 0
}
has a unique solution.
\end{lem}
It follows that equation \eqref{eq: sympl. vort. kazdan-warner} has a unique solution, provided 
\eqst{
\tau > \frac{1}{\text{Vol}(X)}\, \int_X \, \imag\Lambda_{\omega_{\sst X}} F_A \, = \, \frac{2\pi}{\text{Vol}(X)}\deg_{\omega_{\sst X}} P.
}
\end{proof}
WARNING: The Laplacian used by Kazdan and Warner is the negative-definite, which differs from the one we defined here by a minus sign. 

A technical requirement in Lemma \ref{lem: kazdan-warner} is that $B$ be a positive function, outside of a set of measure zero. Apriori, there is no reason to expect this to happen. However, in the following section, we will show that this is always the case, as long as $\tau$ satisfies the condition of Theorem \ref{thm: main thm 1}.

\section{From Symplectic vortices to usual vortices and back}

Let $L \rightarrow X$ be a complex line bundle over $X$ associated to the bundle $P$. As a matter of consistency, we prefer to think of sections of the line bundle as ${\rm U}(1)$-equivariant maps from $P\rightarrow \C$. Let $A$ be a connection on $P$ such that $F_A^{0,2} = 0$. Then $A$ endows $L$ with a holomorphic structure, thus making it a holomorphic line bundle. For an equivariant map $\phi: P \rightarrow \C$ and connection $A$, the $\tau$-vortex equations on $X$ are defined as
\eq{\label{eq: tau vortices}
\left\{
    \begin{array}{lcl}
     \overline{\partial}_{A}\phi = 0 \\      
     \displaystyle \Lambda_{\omega_{\sst X}} F_{A} \, - \, \frac{\imag}{2}\abs{\phi}^2\, + \, \imag\tau  = 0 \\[2mm]
     F^{0,2}_{A} = 0
    \end{array}
  \right.}
The following existence theorem for the $\tau$-vortices was proved by S. Bradlow \cite{bradlow90}.
\begin{thm}
\label{thm: existence of tau vortices}
There exists a solution to the $\tau$-vortex equations \eqref{eq: tau vortices} if and only if $\tau > \frac{2\pi}{\text{Vol}(X)}\deg_{\omega_{\sst X}} P$.
\end{thm}

Notice that the conditions for existence of solutions for both the equations \eqref{eq: symplectic vortices} and \eqref{eq: tau vortices} coincide. We are going to exploit observation to construct a map from the moduli space of symplectic vortices to the moduli space of $\tau$-vortices. In the latter case, the gauge equivalent classes of solutions are in one-to-one correspondence with the set of effective divisors on $X$. This gives us a (surjective) map between the moduli space of symplectic vortices and effective divisors. We show that the effective divisors are given by the zero set of the function $B$, defined in Theorem \ref{thm: main thm 1}.

\begin{thm}
Let $\tau > \frac{2\pi}{\text{Vol}(X)}\deg_{\omega_{\sst X}} P$ and suppose that $(u,A)$ is a solution to \eqref{eq: symplectic vortices}. Then, there exists a unique solution $(\phi, A)$ to \eqref{eq: tau vortices} such that $\mu\circ u = \frac{1}{2} \abs{\phi}^2$.

Conversely, suppose that $(\phi, A)$ is a solution to \eqref{eq: tau vortices}. Then, there exists a solution $(u,A)$ of \eqref{eq: symplectic vortices} such that $\mu\circ u = \frac{1}{2} \abs{\phi}^2$.
\end{thm}

We fix the following notation. By $u$ be denote the equivariant map $P \rightarrow M$ and by $\phi$ we denote the equivariant map $P \rightarrow \C$.

\begin{proof}

Since $(u, A)$ is a solution to \eqref{eq: symplectic vortices}, $\Lambda_{\omega_{\sst X}} F_{A} \, - \, \mu \circ u \, + \tau = 0$. Also, as $\tau$ satisfies the condition $\tau > \frac{2\pi}{\text{Vol}(X)}\deg_{\omega_{\sst X}} P$, by Theorem \ref{thm: existence of tau vortices}, there exists a $\phi$ such that
\eqst{
\Lambda_{\omega_{\sst X}} F_{A} \, - \, \frac{\imag}{2}\abs{\phi}^2\, + \, \imag\tau  = 0.
}
It is easily seen that such a $\phi$ is unique upto the action of the gauge group $\Gg$. To show that this $\phi$ is solution to \eqref{eq: tau vortices}, we need to show that it is holomorphic. To effect this calculation, observe first that $\mu\circ u = \frac{1}{2} \abs{\phi}^2$. Therefore
\eq{
\label{eq: computation 1}
d\left(\mu\circ u \right) = d\mu (D_A u) = d\left(\frac{1}{2} \abs{\phi}^2 \right) = \frac{1}{2} \pair{\phi, D_A \phi}_{\R},
}
where $\pair{\cdot, \cdot}$ denotes the Hermitian inner product on $\C$ and the subscript $\R$ denotes the real part of the product. Equating the $(0,1)$-parts on both the sides we get
\eq{
\label{eq: computation 2}
\frac{1}{2}\pair{\phi, \overline{\partial}_A \phi}_{\R} = d\mu (\overline{\partial}_A u) = 0.
}
So, the $(0,2)$-form $\overline{\partial}\pair{\phi, \overline{\partial}_A \phi}_{\R} = \pair{\overline{\partial}_A \phi \wedge \overline{\partial}_A \phi}_{\R} = 0$, which implies $\abs{\overline{\partial}_A \phi}^2 = 0$. Thus $\phi$ is holomorphic.

To show that the converse statement holds, assume that $(\phi, A)$ is a solution to \eqref{eq: tau vortices}. Since $\tau$ satisfies the condition $\tau > \frac{2\pi}{\text{Vol}(X)}\deg_{\omega_{\sst X}} P$, by Theorem \ref{thm: main thm 1}, there must exist an equivariant map $u: P \rightarrow M$ such that 
\eqst{
\Lambda_{\omega_{\sst X}} F_{A} \, - \, \mu \circ u \, + \tau = 0.
}
By a similar argument as above, we have that $\mu\circ u = \frac{1}{2} \abs{\phi}^2$. In order to show that $(u,A)$ is a solution to \eqref{eq: symplectic vortices}, once again, we must show that $\overline{\partial}_A u = 0$. Repeating the arguments in computations \eqref{eq: computation 1} and \eqref{eq: computation 2}, we get $d\mu\, (\overline{\partial}_A u) = 0$. If $\overline{\partial}_A u$ is not identically zero, then this implies that $\overline{\partial}_A u(p) \in \ker d\mu\,(u(p)) \subset T_{u(p)} M$ for 
every $p \in T_p P$. This in turn implies that $\mu\circ u $ must be identically zero, which is a contradiction, since $\phi \neq 0$. Hence, $\overline{\partial}_A u = 0$ and $(u, A)$ is a solution to \eqref{eq: symplectic vortices}. \vspace{-2mm}

\end{proof}

\section{Map between the moduli spaces}

The above theorem gives an explicit description of the maps between the modui space of symplectic vortices and the moduli space of the usual vortices. Namely, let $\Mm^V(g_{\sst X})$ denote the moduli space of the latter. Then, we have the following map
\eq{
\label{eq: map between moduli}
\Pi: \Mm^{SV}(g_{\sst X}, M) \longrightarrow \Mm^V (g_{\sst X}), ~~~~ [(u,A)] \longmapsto [(\phi, A)], ~~ \text{where} ~~ \mu\circ u = \frac{\abs{\phi}^2}{2}.
}
Since elements of $\Mm^V (g_{\sst X})$ are in one-to-one correspondence with the set of effective divisors on $X$, $\Pi$ maps the the gauge equivalent class of solutions to \eqref{eq: symplectic vortices} to effective divisors. Since $\mu\circ u = \frac{\abs{\phi}^2}{2}$, we see also that the effective divisors are just the zeroes of $\mu\circ u$.

\section{Some Remarks}

\begin{itemize}

\item A four dimensional, Riemannian version of the correspondence \eqref{eq: map between moduli} was studied in \cite{varun17}, where the author studied a generalisation of the Seiberg-Witten equations, obtained by replacing the spinor representation with a hyperK\"ahler manifolds with certain symmetries. This generalisation was introduced by C. Taubes \cite{taubes} in three dimensions and it was extended to four dimensions by V. Pidstrygach \cite{victor}. It was shown that for hyperK\"ahler cones over 3-Sasakians, the generalised equations can be expressed purely in terms of a second order PDE for the hyperK\"ahler moment map. However, Donaldson \cite{donaldson01} had shown that the solutions to the said PDE are in one-to-one correspondence with the solutions to the usual Seiberg-Witten equations. This gives us the map \eqref{eq: map between moduli} between the corresponding moduli spaces.

\item An open question is what is the structure of fibre of the map $\Pi$? 

\item Another interesting open question is whether the results in this article, in a suitable sense, can be extended to a non-Abelian setting. 

\end{itemize}

\bibliographystyle{ieeetr}

\end{document}